\documentclass{amsart}

\usepackage{dutchcal}
\usepackage{amssymb}
\usepackage{amsfonts}
\usepackage{amsthm}
\usepackage{amsmath}
\usepackage{mathtools}
\usepackage{bbm}
\usepackage{bussproofs}
\usepackage{mathrsfs}
\usepackage{tikz, tikz-cd}
\usepackage[all,2cell]{xy}
\usepackage{epigraph}
\usepackage{ stmaryrd } 
\usepackage{todonotes}
\usepackage{comment}

\tikzset{%
	symbol/.style={%
		draw=none,
		every to/.append style={%
			edge node={node [sloped, allow upside down, auto=false]{$#1$}}}
	}
}

\usetikzlibrary{matrix,arrows}

\newtheorem{Theorem}{Theorem}

\newtheorem{proposition}[Theorem]{Proposition}
\newtheorem{lemma}[Theorem]{Lemma}
\newtheorem{corollary}[Theorem]{Corollary}

\theoremstyle{definition}
\newtheorem{example}[Theorem]{Example}
\newtheorem{remark}[Theorem]{Remark}
\newtheorem{definition}[Theorem]{Definition}

\DeclareMathOperator{\Bscr}{\mathscr{B}}
\DeclareMathOperator{\Cscr}{\mathscr{C}}

\DeclareMathOperator{\Pscr}{\mathscr{P}}

\DeclareMathOperator{\Cbb}{\mathbb{C}}

\DeclareMathOperator{\Nbb}{\mathbb{N}}

\DeclareMathOperator{\Rbb}{\mathbb{R}}

\DeclareMathOperator{\Acal}{\mathcal{A}}
\DeclareMathOperator{\Bcal}{\mathcal{B}}
\DeclareMathOperator{\Ccal}{\mathcal{C}}

\DeclareMathOperator{\Lcal}{\mathcal{L}}

\DeclareMathOperator{\Pcal}{\mathcal{P}}

\DeclareMathOperator{\Tcal}{\mathcal{T}}

\DeclareMathOperator{\Meas}{Meas}
\DeclareMathOperator{\Top}{\mathbf{Top}}
\DeclareMathOperator{\Set}{\mathbf{Set}}
\DeclareMathOperator{\R}{\Rbb}
\DeclareMathOperator{\C}{\Cbb}

\DeclareMathOperator{\N}{\Nbb}
\DeclareMathOperator{\Grp}{\mathbf{Grp}}
\DeclareMathOperator{\Ring}{\mathbf{Ring}}
\DeclareMathOperator{\Monad}{\mathbf{Monad}}
\DeclareMathOperator{\Lawvere}{\mathbf{Lawvere}}
\DeclareMathOperator{\inv}{inv}
\DeclareMathOperator{\id}{id}
\DeclareMathOperator{\negg}{neg}
\DeclareMathOperator{\Var}{\mathbf{Var}}
\DeclareMathOperator{\Spec}{Spec}
\DeclareMathOperator*{\colim}{colim}

\let\emptyset\varnothing
\let\epsilon\varepsilon

\numberwithin{Theorem}{section}
\numberwithin{equation}{section}

\usepackage{hyperref} 

\title{Measurable Functions and Topological Algebra}
\author{Geoff Vooys}
\date{\today}

\begin{document}
	
	\begin{abstract}
In this paper we show that if $(X,\mathcal{A})$ is a measurable space and if $Y$ is a topological model of a Lawvere theory $\mathcal{T}$ equipped with $\mathcal{B}$ the Borel $\sigma$-algebra on $Y$, then the set of $\mathcal{B}$-measurable functions from $X$ to $Y$, $\operatorname{Meas}(X,Y)$, is a set-theoretic model of $\mathcal{T}$. As a corollary we give short proofs of the facts that the set of real-valued measurable functions on a measurable space $X$ is a ring and the set of complex-valued measurable functions from $X$ to $\mathbb{C}$ is a ring.
	\end{abstract}
	
	
	\maketitle
	\tableofcontents

	\section{Introduction}
	
	Without much exaggeration, measure theory is as important to modern analysis as the theory of modules is to modern algebra. It is a fundamental tool in functional analysis, statistics, representation theory, harmonic analysis, geometric group theory, and even in number theory. In fact any time a question can be phrased in terms of volumes it probably has a formulation or approach using measure-theoretic tools. However, while measure theory itself is mostly applied in analytic contexts due to its relationship with integration, there is a very set-theoretic algebraic flavour in the foundations and many of the basic results for measure theory. Despite this, it is often an arduous journey for many algebraically-minded graduate students to work through many of the basic measure-theoretic results needed for representation theory and number theory. 
	
	In this paper, however, we show that this difficulty need not be the case. We prove a categorical-algebraic result that gives, if the mathematician writing the proof is willing to be brief and lean on some ``apply Theorem $xyz$ to $\Tcal_{\Ring}$ with the observation...'' in their writing, a one-sentence proof of the fact that for any measurable set $(X,\Acal)$ the set $\Meas(X,\R;\text{Borel})$ of Borel measurable functions is a ring. The trade-off here is that we use some categorical algebraic concepts and prove a much stronger statement using Lawvere theories in order to avoid the analytic perspective and provide a structural proof indicating exactly what is used to prove that $\Meas(X,\R;\text{Borel})$ is a ring. However, this gives an $\epsilon$-free method for proving algebraic properties of sets of measurable functions with values in $\R$.
	
	\subsection{The Structure of the Paper}
	
	In Section \ref{Section: Review Measure} we give a short review of some of the basics of measurable spaces. Because there are many good textbook accounts of measure theory in the literature, we focus instead on when our measurable space $(X,\Acal)$ is the Borel algebra of a topological space $X$. Section \ref{Section: Review Lawvere} gives a quick review of (infinitary) Lawvere theories following the construction of \cite{brandenburg2021large}. The focus we give is more on definitions and examples than in proving and providing general theorems. Finally in Section \ref{Section: Fun} we prove the main theorem of our paper, presented below, and then the following two corollaries.
	
	\begin{Theorem}
		Let $\Tcal$ be a Lawvere theory and let $Y$ be a topological model of $\Tcal$, i.e., $Y \in \Tcal(\Top)_0$. Then if $(X, \Acal)$ is a measurable space, the set
	\[
	\Meas(X,Y;\Bcal_X) = \lbrace f:X \to Y \; | \; f\,\text{is}\,\Bcal_Y\textnormal{-measurable}\rbrace
	\]
	is an object in $\Tcal(\Set)_0$.
	\end{Theorem}
	\begin{corollary}
		Let $R$ be a topological ring. Then for any measurable space $(X,\Acal)$ the set $\Meas(X,R;\Bcal_{R})$ is a ring. In particular, the real-measurable functions $\Meas(X,\R;\Bcal_{\R})$ and complex-measurable functions $\Meas(X,\C;\Bcal_{\C})$ are rings.
	\end{corollary}
	\begin{corollary}
		Let $G$ be a topological group and let $(X,\Acal)$ be any measurable space. Then $\Meas(X,G;\Bcal_{G})$ is a group. In particular, for any local field $F$ the set $\Meas(X,\operatorname{GL}_n(F);\Bcal_{\operatorname{GL}_n(F)})$ is a group.
	\end{corollary}
	
	\subsection{Acknowledgments}
	This work was supported by an AARMS postdoctoral fellowship.
	
	\section{A Review of Measure Theory for Topological Spaces}\label{Section: Review Measure}
	
	In this short section we review some of the basics of measure theory regarding measurable spaces and $\sigma$-algebras. Most of this section is simply review of what is explained in more detail in many textbooks in the literature (cf.\@ \cite{conway2012course}, \cite{Doob}, \cite{Tao}, \cite{taylormeasure}, for instance), but we present it primarily because, in my opinion, Example \ref{Example: sigma algebras not topologies} and the interaction between uncountably infinite topological spaces and measure theory are not as well-known as they likely should be.
	
	\begin{definition}\label{Defn: Sigma Alg}
	Let $X$ be a set and let $\Pcal X$ denote the power set
	\[
	\Pcal X := \lbrace S \; | \; S \subseteq X \rbrace
	\]
	of $X$. Then a $\sigma$-algebra on $X$ is a subset $\Acal \subseteq \Pcal X$ for which:
	\begin{itemize}
		\item We have that $X \in \Acal$.
		\item If $A \in \Acal$, then $A^c := X \setminus A \in \Acal$;
		\item If $I$ is a countable index set and if $\lbrace A_i \; | \; i \in I \rbrace \subseteq \Acal$ then
		\[
		\bigcup_{i \in I} A_i \in \Acal.
		\]
	\end{itemize}
	\end{definition}
	\begin{definition}\label{Defn: Measurable Space}
	A measurable space is a pair $(X,\Acal)$ where $X$ is a set and $\Acal$ is a $\sigma$-algebra on $X$.
	\end{definition}
	Of course the definition of a measurable space is geometrically motivated; it allows us to axiomitize which pieces of a set $X$ can have measured volume and gives us inference rules for what we can say about how these sets must behave. The first axiom says that if we can measure the volume of a piece $A \subseteq X$ then we can measure the volume of the complement of $A$, while the second rule says that if we can measure the volume of countably many pieces $A_i \subseteq X$ then we can measure the volume of the compositum $\cup_{i \in I} A_i$. Another important way to describe the geometry of a set $X$ is through a topology on $X$, which we now describe.
	
	\begin{definition}\label{Defn: Topology}
		A topological space is a pair $(X,\Tcal)$ where $X$ is a set and $\Tcal \subseteq \Pcal X$ is a set of subsets we declare to be open such that:
		\begin{itemize}
			\item $\emptyset, X \in \Tcal$;
			\item For any set $I$, if $\lbrace U_i \; | \; i \in I, U_i \in \Tcal \rbrace \subseteq \Tcal$ then
			\[
			\bigcup_{i \in I} U_i \in \Tcal.
			\]
			\item For any finite collection $U_1, \cdots, U_n \in \Tcal$, there is a $U \in \Tcal$ for which
			\[
			\bigcap_{i=1}^{n} U_i^{c} = U^c.
			\]
		\end{itemize}
	\end{definition}
	An important consequence of the definition of a $\sigma$-algebra is that such a set is closed with respect to countable intersections of its members. This allows us to deduce the following proposition which tells us that $\sigma$-algebras are topologies on countable sets.
	\begin{proposition}
	Let $(X,\Acal)$ be a countable measure space. Then $\Acal$ is a topology on $X$.
	\end{proposition}
	With this in mind it is tempting to think that $\sigma$-algebras give rise to topologies on sets, but sadly it is relatively well-known that this need not be the case. However, as this is not as well-known as it should be, we present the example below to illustrate this fact.
\begin{example}\label{Example: sigma algebras not topologies}
	Let $X$ be an uncountably infinite set and let $\Acal \subseteq \Pcal(X)$ be the $\sigma$-algebra of countable or cocountable subsets of $X$, i.e.,
	\[
	\Acal := \lbrace Y \subseteq X \; : \; \lvert Y \rvert \leq \aleph_0\,\text{or}\,\lvert Y^c \rvert \leq \aleph_0 \rbrace.
	\]
	It is routine and tedious to check that $\Acal$ is a $\sigma$-algebra on $X$. To see that it does not generate a topology on $X$ note that since the set $\lbrace x \rbrace \subseteq X$ for each $x \in X$, $\Acal$ contains every singleton point of $X$ and so must be the discrete topology if it is a topology. However, as $\Acal$ contains no uncountable set with uncountable complement this is impossible and so $\Acal$ cannot be a topology on $X$.
\end{example}
	
\section{A Review of Lawvere Theories}\label{Section: Review Lawvere}
In this section we give a short review of Lawvere theories, their basic notions, and some examples. Lawvere theories are generalizations and formalizations of algberiac theories (as they appear in universal algbera and logic; cf.\@ ) and give us a language to discuss and the infinitary Lawvere theories we allow in this paper even provide a way to discuss algebraic objects that have infinitely many operations. We follow \cite[Appendix A]{brandenburg2021large}, which contains an expository account of Lawvere theories in detail, and formalize the notion of Lawvere theories as certain limit sketches; other equivalent formalizations are given in \cite{TTT}, \cite{HylandPower}, and \cite{AdamekRosicky}, among other sources. Of course, it is an important theorem of Linton (cf.\@ \cite{Linton} for the usual reference and \cite{brandenburg2021large} for a self-contained explicit and direct proof) that there is an equivalence of categories
\[
\Lawvere \simeq \Monad(\Set).
\]
While we will not need more than a basic definition of Lawvere theories in this paper, we present some examples in this short section so that we can firmly ground our intuition in the sort of examples which are likely to be most familiar.

\begin{definition}
An infintary Lawvere theory is a pair $(\Lcal, X)$ where $\Lcal$ is a locally small category with all $\Set$-indexed products and $X$ is a distinguished object of $\Lcal$ such that for all objects $Y$ of $\Lcal$ there exists a set $I$ and an isomorphism
\[
Y \cong \prod_{i \in I} X = X^I
\]
in $\Lcal$.
\end{definition}
\begin{remark}\label{Remark: Lazy man}
Because of the nature of a Lawvere theory $(\Lcal, X)$ described above, to describe functors $F:\Lcal \to \Cscr$ it suffices to define $F$ on the objects $X^I$ for all sets $I$ and on all morphisms between all $X^I$ and $X^J$.
\end{remark}
\begin{definition}
We say that a Lawvere theory $(\Lcal,X)$ is finitary if for any index sets $I, J$ we have that
\[
\Lcal(X^I, X^J) \cong \left(\colim_{\substack{E \subseteq I \\ E\,\text{finite}}} \Lcal(X^E,X)\right)^I.
\]
\end{definition}
\begin{example}\label{Example: Lawvere theory of gorups}
The (finitary) Lawvere theory of groups is the category $\Tcal_{\Grp}$ where we define the obejcts to be given by:
\begin{itemize}
	\item For every index set there is an object $G^I$. The distinguished object is $G := G^{[0]}$ where $[0] = \lbrace 0 \rbrace$. We also write $\top := G^{\emptyset}$.
\end{itemize}
We define the morphisms to be generated by making sure that each $G^I$ is a product indexed by $I$ together with some extra morphisms
\begin{itemize}
	\item $1_G:\top \to G$
	\item $\inv:G \to G$
	\item $\mu:G \times G \to G$
\end{itemize}
which satisfy the diagrams
\[
\begin{tikzcd}
G \ar[r]{}{\cong} \ar[d, swap]{}{\cong} \ar[drr, equals]{}{} & G \times \top \ar[r]{}{\id_G \times 1_G} &  G \times G \ar[d]{}{\mu} \\
\top \times G \ar[r, swap]{}{1_G \times \id_G} & G \times G \ar[r, swap]{}{\mu} & G
\end{tikzcd}
\]
\[
\begin{tikzcd}
(G \times G) \times G \ar[rr]{}{\mu \times \id_G} \ar[d, swap]{}{\cong} & & G \times G \ar[d]{}{\mu} \\
G \times (G \times G) \ar[dr, swap]{}{\id_G\times \mu} & & G \\
&  G \times G \ar[ur, swap]{}{\mu}
\end{tikzcd}
\]
and:
\[
\begin{tikzcd}
G \ar[r]{}{\Delta} \ar[dr, swap, dashed]{}{!_G} & G \times G \ar[rr, shift left = .5ex]{}{\inv \times \id_G} \ar[rr, swap, shift right = .5ex]{}{\id_G\times \inv} & & G \times G \ar[d]{}{\mu} \\
 & \top \ar[rr, swap]{}{1_G} & & G
\end{tikzcd}
\]
\end{example}
\begin{example}\label{Example: Lawvere of associative rings}
	The (finitary) Lawvere theory of (associative unital) rings is the category $\Tcal_{\Ring}$ where we define the obejcts to be given by:
	\begin{itemize}
		\item For every index set there is an object $G^I$. The distinguished object is $R := R^{[0]}$ where $[0] = \lbrace 0 \rbrace$. We also write $\top := R^{\emptyset}$.
	\end{itemize}
	We define the morphisms to be generated by making sure that each $R^I$ is a product indexed by $I$ together with some extra morphisms
	\begin{itemize}
		\item $0_R:\top \to R$
		\item $1_R:\top \to R$
		\item $\negg:R \to R$
		\item $\alpha:R \times R \to R$
		\item $\mu:R \times R \to R$
	\end{itemize}
	which satisfy the diagrams
	\[
	\begin{tikzcd}
	R \ar[r]{}{\cong} \ar[d, swap]{}{\cong} \ar[drr, equals]{}{} & R \times \top \ar[r]{}{\id_R \times 1_R} &  R \times R \ar[d]{}{\mu} \\
	\top \times R \ar[r, swap]{}{1_R \times \id_R} & R \times R \ar[r, swap]{}{\mu} & R
	\end{tikzcd}
	\]
	\[
	\begin{tikzcd}
	(R \times R) \times R \ar[rr]{}{\mu \times \id_R} \ar[d, swap]{}{\cong} & & R \times R \ar[d]{}{\mu} \\
	R \times (R \times R) \ar[dr, swap]{}{\id_R\times \mu} & & R \\
	&  R \times R \ar[ur, swap]{}{\mu}
	\end{tikzcd}
	\]
	\[
	\begin{tikzcd}
	R \ar[r]{}{\cong} \ar[d, swap]{}{\cong} \ar[drr, equals]{}{} & R \times \top \ar[r]{}{\id_R \times 0_R} &  R \times R \ar[d]{}{\alpha} \\
	\top \times R \ar[r, swap]{}{0_R \times \id_R} & R \times R \ar[r, swap]{}{\alpha} & R
	\end{tikzcd}
	\]
	\[
	\begin{tikzcd}
	(R \times R) \times R \ar[rr]{}{\alpha \times \id_R} \ar[d, swap]{}{\cong} & & R \times R \ar[d]{}{\alpha} \\
	R \times (R \times R) \ar[dr, swap]{}{\id_R\times \alpha} & & R \\
	&  R \times R \ar[ur, swap]{}{\alpha}
	\end{tikzcd}
	\]
	\[
	\begin{tikzcd}
	R \ar[r]{}{\Delta} \ar[dr, swap, dashed]{}{!_G} & R \times R \ar[rr, shift left = .5ex]{}{\negg \times \id_R} \ar[rr, swap, shift right = .5ex]{}{\id_R\times \negg} & & R \times R \ar[d]{}{\alpha} \\
	& \top \ar[rr, swap]{}{0_R} & & R
	\end{tikzcd}
	\]
	and diagrams expressing left and right distributivity of multiplication and addition:
	\[
	\begin{tikzcd}
	R \times (R \times R) \ar[d, swap]{}{\Delta} \ar[rrr]{}{\id_R\times \alpha} &  & & R \times R \ar[d]{}{\mu} \\
	(R \times R) \times (R \times R) \ar[d, swap]{}{\cong} & (R \times R) \times (R \times R) \ar[r, swap]{}{\mu \times \mu} & R \times R \ar[r]{}{\alpha} & R \\
	R \times (R \times R) \times R \ar[r, swap]{}{\id_R \times \text{switch} \times \id_R} & R \times (R \times R) \times R \ar[u, swap]{}{\cong} 
	\end{tikzcd}
	\]
	\[
	\begin{tikzcd}
	(R \times R) \times R \ar[d, swap]{}{\Delta} \ar[rrr]{}{\alpha \times \id_R} & & & R \times R \ar[d]{}{\mu} \\
	(R \times R) \times (R \times R) \ar[d, swap]{}{\cong} & (R \times R) \times (R \times R) \ar[r]{}{\mu \times \mu}& R \times R \ar[r, swap]{}{\alpha} & R  \\
	R \times (R \times R) \times R \ar[r, swap]{}{\id_R \times \text{switch} \times \id_R} & R \times (R \times R) \times R \ar[u, swap]{}{\cong}	\end{tikzcd}
	\]
\end{example}
\begin{definition}
Let $\Cscr$ be a category and let $(\Lcal,X)$ be a Lawvere theory. A $\Cscr$-model of $\Lcal$ is a product-preserving functor $M:\Lcal \to \Cscr$. In this case we call $M(X)$ the underlying object of the model. Moreover, the category of all such models is the category
\[
\Lcal(\Cscr) := [\Lcal,\Cscr]_{\text{prod}}
\]
of product-preserving functors $\Lcal \to \Cscr$ and their natural transformations.
\end{definition}
\begin{remark}
If $(\Lcal,X)$ is a Lawvere theory then:
\begin{itemize}
	\item A $\Set$-valued model of $\Lcal$ is often called ``an $\Lcal$'' by abuse of notation (just as we call models of $\Tcal_{\Grp}$ in $\Set$ groups, $\Tcal_{\Ring}$ in $\Set$ rings, etc.). There are  equivalences of categories $\Tcal_{\Grp}(\Set) \simeq \Grp$ and $\Tcal_{\Ring}(\Set) \simeq \Ring$ which follow immediately after unwinding the definitions of everything in sight.
	\item A $\Top$-valued model of $\Lcal$ is a topological model of $\Lcal$. Classical examples of this include topological groups (models of $\Tcal_{\Grp}$) and topological rings (models of $\Tcal_{\Ring}$). However, topological fields are \textbf{not} examples of $\Top$-valued models of a Lawvere theory because the theory of fields is not a Lawvere theory.
	\item If $K$ is a field and if $\Cscr = \Var_{/\Spec K}$ is the category of $K$-varieties then $\Lcal(\Var_{/\Spec K})$ is the category of $K$-variety models of $\Lcal$. When $\Lcal = \Tcal_{\Grp}$, $\Lcal(\Var_{/\Spec K}) = \Tcal_{\Grp}(\Var_{/\Spec K})$ is the category of algebraic groups over $K$ (which is a category whose objects occupy an area of central importance in the Langlands Programme and in algebraic geometry).
\end{itemize}
\end{remark}
	
\section{Putting it All Together: Borel Measures}\label{Section: Fun}

We now discuss some topological measure theory before proceeding to prove our main theorem. This material is relatively well-known, so we will mostly focus on giving a quick description in order to get to show how topological models of Lawvere theories and Borel algebras interact.
\begin{definition}
Let $X$ be a topological space. We define the Borel algebra on $X$, $\Bcal_X$, to be the $\sigma$-algebra generated by the open subsets $U \subseteq X$.
\end{definition}
\begin{definition}
Let $X$ be a topological space. We say that a $\sigma$-algebra $\Acal$ on $X$ is a Borel algebra if for any open set $U \subseteq X$, $U \in \Acal$.
\end{definition}

From the definition it is immediate that if $X$ is a topological space then we can describe the Borel algebra $\Bcal_X$ by
\[
\Bcal_X := \bigcap_{\substack{\Acal\,\text{is a}\, \sigma\text{-algebra} \\ \Acal\,\text{is a Borel algebra}}} \Acal.
\]
This follows more or less immediately from the description of $\Bcal_X$ being generated by the open sets of $X$.

\begin{proposition}\label{Prop: Smallest sigma alg with all opens}
Let $X$ be a topological space. The Borel algebra $\Bcal_X$ on $X$ admits the description
\[
\Bcal_X = \inf\left\lbrace \Acal \; | \; \Acal\,\textnormal{is a}\;\sigma\textnormal{-algebra}, U \in \Acal, U \subseteq X\,\textnormal{open} \right\rbrace
\]
in the power set $\Pscr(X)$ regarded as a Boolean algebra with respect to the subset order $\subseteq$.
\end{proposition}

With this terminology we can now develop the theory of Borel measurable functions in order to derive our main result. These are the measurable functions $f:X \to Y$ where $Y$ is a topological space which interact in a compatible way with the topology on $Y$.
\begin{definition}
Let $Y$ be a topological space and let $(X,\Acal)$ be a measurable space. We then say that $f:X \to Y$ is Borel measurable (or $\Bcal_Y$-measurable if we need to emphasize the space $Y$) if and only if for all $B \in \Bcal_Y$ we have $f^{-1}(B) \in \Acal$.
\end{definition}

This leads to a standard but helpful set-theoretic lemma.
\begin{lemma}\label{Lemma: Set of preimage borels is sigma alg}
If $(X,\Acal)$ is a measurable space and if $f:X \to Y$ is a function of sets then
\[
\Bcal := \left\lbrace B \subseteq Y \; | \; f^{-1}(B) \in \Acal \right\rbrace
\]
is a $\sigma$-algebra on $Y$.
\end{lemma}
\begin{proof}
We proceed with the observations that for any arbitrary family $\lbrace B_i \; | \; i \in I, B_i \subseteq Y \rbrace$ of subsets,
\[
f^{-1}\left(\bigcup_{i \in I} B_i\right) = \lbrace x \in X \; | \; \exists\,i \in I.\,f(x) \in B_i \subseteq Y \rbrace = \bigcup_{i \in I} f^{-1}(B_i)
\]
and dually
\[
f^{-1}\left(\bigcap_{i \in I} B_i\right) = \bigcap_{i \in I}f^{-1}(B_i).
\]
Thus if we have any countable collection of sets $B_n \subseteq Y$ for which $f^{-1}(B_n) \in \Acal$ for all $n \in \N$ we get that
\[
f^{-1}\left(\bigcup_{n \in \N}B_n\right) = \bigcup_{n \in \N} f^{-1}(B_n) \in \Acal
\]
and
\[
f^{-1}\left(\bigcap_{n \in \N}B_n\right) = \bigcap_{n \in \N}f^{-1}\left(B_n\right) \in \Acal
\]
because $\Acal$ is a $\sigma$-algebra and each of $f^{-1}(B_n) \in \Acal$.

To verify that $\Bcal$ is closed under complements, note that $f^{-1}(Y) = X$ and $f^{-1}(S \setminus T) = f^{-1}(S) \setminus f^{-1}(T)$ for any $S, T \subseteq Y$. Then if $B \in \Bcal$ we get that, because $\Acal$ is a $\sigma$-algebra,
\[
f^{-1}(B^c) = f^{-1}(Y \setminus B) = f^{-1}(Y) \setminus f^{-1}(B) = X \setminus f^{-1}(B) = f^{-1}(B)^c \in \Acal.
\]
Thus $f^{-1}(B)^c \in \Bcal$ and so $\Bcal$ is a $\sigma$-algebra.
\end{proof}

We now can establish some basic properties about the Borel measurability of continuous functions and how these functions post-compose with other Borel measurable functions.
\begin{lemma}\label{Lemma: Continuous functions Borel Measurable}
Let $X$ and $Y$ be topological spaces and regard $X$ as a measurable space by equipping it with its Borel measure $\Bcal_{X}$. If $f:X \to Y$ is a continuous function of topological spaces then $f$ is Borel measurable.
\end{lemma}
\begin{proof}
Since $f$ is continuous, for every open set $V \subseteq Y$ we have that $f^{-1}(V) \subseteq X$ is open and so $f^{-1}(V) \in \Bcal_X$ as well. However, this then implies that every open $V \subseteq Y$ open we get that 
\[
f^{-1}(V) \in \lbrace U \subseteq Y \; | \; f^{-1}(U) \in \Bcal_X \rbrace =: \Ccal.
\] 
Because $\Ccal$ is a $\sigma$-algebra by Lemma \ref{Lemma: Set of preimage borels is sigma alg} which contains every open set of $Y$, we see that $\Bcal_Y \subseteq \Ccal$ by Proposition \ref{Prop: Smallest sigma alg with all opens}. However, this implies that for any $A \in \Bcal_X$ we have $f^{-1}(A) \in \Bcal_X$ and so $f$ is Borel measurable.
\end{proof}
\begin{proposition}\label{Prop: Composite of borel meas functions}
Let $(X,\Acal)$ be a measurable space and let $Y$ and $Z$ be topological spaces equipped with their corresponding Borel $\sigma$-algebras $\Bcal_Y$ and $\Bcal_Z$. Then if $f:X \to Y$ is $\Bcal_Y$-measurable and if $g:Y \to Z$ is $\Bcal_Z$-measurable, the composite $g \circ f:X \to Z$ is $\Bcal_Z$-measurable.
\end{proposition}
\begin{proof}
Let $V \subseteq Z$ with $Z \in \Bcal_Z$; it now suffices to show that
\[
(g \circ f)^{-1}(Z) = f^{-1}(g^{-1}(V))
\]
is contained in $\Acal$. However, as $g$ is $\Bcal_Z$-measurable, $g^{-1}(V) \in \Bcal_Y$. Similarly, as $f$ is $\Bcal_Y$-measurable, $f^{-1}(g^{-1}(V)) \in \Acal$. Thus
\[
f^{-1}\left(g^{-1}(V)\right) = (g \circ f)^{-1}(V) \in \Acal
\]
and so $g \circ f$ is $\Bcal_Z$-measurable.
\end{proof}
\begin{corollary}\label{Cor: Composite of mesurable function and continuous function}
Let $(X,\Acal)$ be a measurable space and let $Y$ and $Z$ be topological spaces equipped with their corresponding Borel algebras $\Bcal_Y$ and $\Bcal_Z$. Then  if $g:Y \to Z$ is a continuous function the composite function $g \circ f:X \to Z$ is $\Bcal_Z$-measurable.
\end{corollary}
\begin{proof}
Apply Lemma \ref{Lemma: Continuous functions Borel Measurable} to derive that $g$ is Borel-measurable and then use Proposition \ref{Prop: Composite of borel meas functions}.
\end{proof}

We now have two last technical ingredients to provide. These show the Borel measurabliity of situations we'll find ourselves in when studying the measurability of topological models of a Lawvere theory $\Lcal$.
\begin{lemma}\label{Lemma: Pairing map of measurable functions is meas}
Let $(X,\Acal)$ be a measurable space, let $I$ be an index set, and let
\[
\left\lbrace f_i:X \to Y \; | \; i \in I \right\rbrace
\]
be a collection of functions which are all $\Bcal_Y$-measurable. If $F:X \to Z^I$ is the map given by
\[
f(x) := \big(f_i(x)\big)_{i \in I}
\]
then $f$ is $\Bcal_{Y^I}$-measurable.
\end{lemma}
\begin{proof}
Because $Y^I$ is the categorical product of $Y$ taken $I$-many times and carries the product topology, $Y^I$ has basis of open sets
\[
\Bscr := \left\lbrace \prod_{i \in I} U_i \; : \; U_i \subseteq Y\,\text{open}, U_i \ne Y\,\text{for all but finitely many}\,i \right\rbrace.
\]
Similarly, because the Borel $\sigma$-algebra $\Bcal_{Y^I}$ is generated by the opens of $Y^I$ and the open sets of $Y^I$ are generated by those in $\Bscr$ it suffices to prove that for any basic open $U = \prod_{i \in I} U_i$ in $\Bscr$, $f^{-1}(V) \in \Acal$. However, we now note that
\begin{align*}
	f^{-1}\left(\prod_{i \in I} U_i\right) &= \left\lbrace x \in X \; : \; (f_i(x))_{i \in I} \in \prod_{i  \in I} U_i \right\rbrace = \left\lbrace x \in X \; | \; \forall\,i\in I.\,f_i(x)\in U_i\right\rbrace \\
	&= \bigcap_{i \in I} f_i^{-1}(U_i).
\end{align*}
Because $U$ is a basic open in $Y^I$, at most finitely many of the $U_i$ satisfy $U_i \ne Y$. But then because $f_i^{-1}(Y) = X$, the intersection above is a finite intersection
\[
f^{-1}\left(\prod_{i \in I} U_i\right) = \bigcap_{U_i \ne Y} f_i^{-1}(U_i);
\]
moreover since each $U_i$ is open and each $f_i$ is $\Bcal_Y$-measurable, $f_i^{-1}(U_i) \in \Acal$. But then, as finite intersections are countable, it follows that
\[
f^{-1}\left(\prod_{i \in I} U_i\right) = \bigcap_{U_i \ne Y} f_i^{-1}(U_i) \in \Acal
\]
and so $f$ is Borel measurable.
\end{proof}
\begin{remark}
The lemma above says in categorical terms that if we have a collection of $\Bcal_Y$-measurable morphisms
\[
\left\lbrace f_i:X \to Y \; | \; i \in I \right\rbrace
\]
then the pairing map
\[
\langle f_i \rangle_{i \in I}:X \to Y^I
\]
is $\Bcal_{Y^I}$-measurable.
\end{remark}
\begin{proposition}\label{Prop: Paring and composition}
Let $(X,\Acal)$ be a measurable space, let $Y$ be a topological space, let $I$ be an index set. If
\[
\lbrace f_i:X \to Y \; | \; i \in I \rbrace
\]
is a collection of $\Bcal_Y$-indexed functions and $g:Y^I \to Y$ is a $\Bcal_Y$-measurable function, then the function $h:X \to Y$ given by $h = g \circ \langle f_i \rangle_{i \in I}$ is $\Bcal_Y$-measurable.
\end{proposition}
\begin{proof}
Simply apply Lemma \ref{Lemma: Pairing map of measurable functions is meas} and Proposition \ref{Prop: Composite of borel meas functions} to the composition $h = g \circ \langle f_i \rangle_{i \in I}$.
\end{proof}

Combining these tools together allows us to prove our main theorem: the collection of Borel-measurable functions from a measurable space $(X,\Acal)$ into a topological model of a Lawvere theory $\Lcal$ gives rise to a set-theoretic model of $\Lcal$.

	\begin{Theorem}\label{Thm: Main thm}
	Let $\Tcal$ be a Lawvere theory and let $Y$ be a topological model of $\Tcal$, i.e., $Y \in \Tcal(\Top)_0$. Then if $(X, \Acal)$ is a measurable space, the set
	\[
	\Meas(X,Y;\Bcal_X) = \lbrace f:X \to Y \; | \; f\,\text{is}\,\Bcal_Y\textnormal{-measurable}\rbrace
	\]
	is an object in $\Tcal(\Set)_0$.
	\end{Theorem}
	\begin{proof}
	Because $Y$ is a topological model of $\Tcal$, we write
	\[
	\underline{Y}:\Tcal \to \Top
	\]
	for the corresponding functor and specify that $\underline{Y}(D) = Y$ for $D$ the distinguished object of $\Tcal$. To show that $\Meas(X,Y;\Bcal_Y)$ is a model of $\Tcal$ in $\Set$, we define
	\[
	\underline{\Meas(X,Y;\Bcal_Y)}:\Tcal \to \Set
	\]
	as follows:
	\begin{itemize}
		\item For objects $D^I$ for index sets $I$, we define
		\[
		\underline{\Meas(X,Y;\Bcal_Y)}(D^I) := \Meas(X,Y^{I};\Bcal_{Y^I}).
		\]
		\item For morphisms $\phi:D^I \to D^J$ in $\Tcal$, we define
		\[
		\underline{\Meas(X,Y;\Bcal_Y)}(D^I) \to \underline{\Meas(X,Y;\Bcal_{Y})}(D^J)
		\] 
		to be the function
		\[
		\Meas(X,Y^I;\Bcal_{Y^I}) \to \Meas(X,Y^{J};\Bcal_{Y^J})
		\]
		given by $f \mapsto \underline{Y}(\varphi) \circ f$.
	\end{itemize}
	That this is a functor is a straightforward check from definition and Remark \ref{Remark: Lazy man}, while the fact that the assignment on morphisms is well-defined follows from Proposition \ref{Prop: Paring and composition}. To see that it is product preserving, let $I$ be a set and assume that for all $i \in I$ there is a map $f_i:Z \to \Meas(X,Y;\Bcal_Y)$. Consider the functions $p_i:\Meas(X,Y^{I};\Bcal_{Y^I}) \to \Meas(X,Y;\Bcal_Y)$ given by post-composition by the projection maps $\pi_i:Y^I \to Y$; these are the maps we must show play the role of projections. We define a map $f:Z \to \Meas(X,Y^{I};\Bcal_{Y^I})$ by, for each $z \in Z$, letting $f(z)$ be the function $X \to Y^I$ given by the pairing $\langle f_i(z) \rangle_{i \in I}$. It then follows by construction that for all $i \in I$
	\[
	p_i \circ f = \pi_i \circ \langle f_i \rangle_{i \in I} = f_i
	\]
	so the diagram
	\[
	\begin{tikzcd}
	Z \ar[d, swap]{}{f} \ar[dr]{}{f_i} \\
	\Meas(X,Y^I;\Bcal_{Y^I}) \ar[r, swap]{}{p_i} & \Meas(X,Y;\Bcal_{Y})
	\end{tikzcd}
	\]
	commutes. The fact that $f$ is unique follows from the fact that each pairing map $\langle f_i(z) \rangle_{i \in I}$ is induced by the universal property of the (Cartesian) product. Thus $\Meas(X,Y^I;\Bcal_{Y^I}) \cong \Meas(X,Y;\Bcal_{Y})^I$ and so the functor $\underline{\Meas(X,Y;\Bcal_{Y})}:\Tcal \to \Set$ is product-preserving. Consequently $\Meas(X,Y;\Bcal_{Y})$ is a set-theoretic model of the Lawvere theory $\Tcal$.
	\end{proof}
	
	The usefulness of this theorem is exemplified by following straightforward algebraic proofs of the corollaries below (which are in particular of use in functional analysis, representation theory, harmonic analysis, and number theory).
	\begin{corollary}\label{Cor: Top ring is ring in set}
	Let $R$ be a topological ring. Then for any measurable space $(X,\Acal)$ the set $\Meas(X,R;\Bcal_{R})$ is a ring. In particular, the real-measurable functions $\Meas(X,\R;\Bcal_{\R})$ and complex-measurable functions $\Meas(X,\C;\Bcal_{\C})$ are rings.
	\end{corollary}
	\begin{proof}
	Since rings are set-theoretic models of the Lawvere theory $\Tcal_{\Ring}$ of rings described in Example \ref{Example: Lawvere of associative rings} and topological rings are topological models of $\Tcal_{\Ring}$, we simply apply Theorem \ref{Thm: Main thm} to the set $\Meas(X,R;\Bcal_{R})$. The final statements follow from the fact that $\R$ and $\C$ are topological rings.
	\end{proof}
	\begin{corollary}\label{Cor: Top group is group in set}
	Let $G$ be a topological group and let $(X,\Acal)$ be any measurable space. Then $\Meas(X,G;\Bcal_{G})$ is a group. In particular, for any local field $F$ the set $\Meas(X,\operatorname{GL}_n(F);\Bcal_{\operatorname{GL}_n(F)})$ is a group.
	\end{corollary}
	\begin{proof}
	This follows mutatis mutandis to the proof of Corollary \ref{Cor: Top ring is ring in set} save with the Lawvere theory $\Tcal_{\Ring}$ replaced with the Lawvere theory $\Tcal_{\Grp}$ of groups as described in Example \ref{Example: Lawvere theory of gorups}. The final statement of the corollary follows from the fact that because $F$ is a local field $\operatorname{GL}_n(F)$ naturally inherits the structure of a topological group from the topology on $F$.
	\end{proof}

\bibliographystyle{amsplain}
\bibliography{MeasureBib}	
	
\end{document}